\documentclass[10pt]{article}
\usepackage{extsizes,graphicx,amssymb,amsmath,amsthm,amsfonts,mathtext, color,wrapfig,ytableau,mathtools}

\def\be{\begin{eqnarray}}
\def\ee{\end{eqnarray}}

\newtheorem{theorem}{Theorem}
\newtheorem{definition}[theorem]{Definition}
\newtheorem{lemma}[theorem]{Lemma}
\newtheorem{proposition}[theorem]{Proposition}
\newtheorem{remark}[theorem]{Remark}
\newtheorem{corollary}[theorem]{Corollary}

\newtheorem{innercustomprop}{Proposition}
\newenvironment{customprop}[1]
  {\renewcommand\theinnercustomprop{#1}\innercustomprop}
  {\endinnercustomprop}

\hoffset= - 0.8in         

\begin{document}

\thispagestyle{empty}
\baselineskip14pt

\ytableausetup{boxsize=5pt}
\hfill ITEP/TH-11/17
\par
\hfill IITP/TH-5/17

\bigskip
\bigskip

\centerline{\Large{Genus Two Generalization of $A_1$ spherical DAHA
}}

\vspace{3ex}

\centerline{\large{\emph{S.Arthamonov\footnote{Department of Mathematics, Rutgers, The State University of New Jersey; semeon.artamonov@rutgers.edu \\ ITEP, Moscow, Russia} and Sh.Shakirov\footnote{Harvard Society of Fellows, Harvard University, Cambridge MA; shakirov@fas.harvard.edu \\
Institute for Information Transmission Problems, Moscow, Russia}}}}

\vspace{3ex}

\centerline{ABSTRACT}

\bigskip

{\footnotesize
We consider a system of three commuting difference operators in three variables $x_{12},x_{13},x_{23}$ with two generic complex parameters $q,t$. This system and its eigenfunctions generalize the trigonometric $A_1$ Ruijsenaars-Schneider model and $A_1$ Macdonald polynomials, respectively. The principal object of study in this paper is the algebra generated by these difference operators together with operators of multiplication by $x_{ij} + x_{ij}^{-1}$. We represent the Dehn twists by outer automorphisms of this algebra and prove that these automorphisms satisfy all relations of the mapping class group of the closed genus two surface.
Therefore we argue from topological perspective this algebra is a genus two generalization of $A_1$ spherical DAHA.
}

\section{Introduction}

Double affine Hecke algebra \cite{DAHA}, introduced and developed by I.Cherednik, plays a crucial role in modern representation theory and mathematical physics from a number of perspectives (see \cite{DAHA} for a comprehensive overview). One perspective which is especially interesting for us is its close relation with topology of the torus and its mapping class group $SL(2,{\mathbb Z})$: namely, this mapping class group acts on the DAHA by outer automorphisms. This, in particular, allows to compute DAHA-Jones polynomials of torus \cite{DAHAJones} and iterated torus \cite{DAHAJones2} links, which generalize and improve WRT invariants \cite{W,RT} and hence are very interesting from the viewpoint of knot theory. DAHA-Jones polynomials are also related to a number of subjects in mathematical physics, such as refined Chern-Simons theory \cite{AS} and knot superpolynomials \cite{superpoly1,superpoly2}.

We concentrate on a particular subalgebra of the double affine Hecke algebra, called spherical DAHA, in the case of the $A_1$ root system. In this case, the spherical DAHA admits a very simple polynomial representation as the algebra generated by two operators,

\begin{align}
{\hat {\cal O}}_B = x + x^{-1}, \ \ \ \ \ \ {\hat {\cal O}}_A = \dfrac{x^{-1} - t x}{x^{-1} - x} \ {\hat \delta} + \dfrac{x - t x^{-1}}{x - x^{-1}} \ {\hat \delta}^{-1}
\label{OAOB}
\end{align}
\vspace{1ex}
\begin{wrapfigure}{r}{0.35\textwidth}
\vspace{-1ex}
  \begin{center}
    \includegraphics[width=0.35\textwidth]{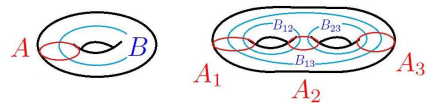}
  \end{center}
  \vspace{-2ex}
  \caption{A choice of $A$- and $B$-cycles.}
  \label{Cycles}
  \vspace{-1ex}
\end{wrapfigure}
acting on the space of symmetric Laurent polynomials in a single variable $x$; here ${\hat \delta} f(x) = f\left(q^{\frac{1}{2}} x\right)$ is a shift operator. These two operators are associated with the $A$- and $B$-cycles of the torus, as shown in Figure 1, and the mapping class group acts on the algebra generated by these operators by certain explicit outer automorphisms.

In this paper we consider an algebra generated by six operators acting on the space of Laurent polynomials in three variables $x_{12},x_{13},x_{23}$, associated with the three $A$- and three $B$-cycles of a genus two surface, as shown on Figure 1. We represent the Dehn twists by outer automorphisms of this algebra and prove that these automorphisms satisfy all relations of the mapping class group of a closed genus two surface. This suggests this algebra is a genus two generalization of $A_1$ spherical DAHA. The reason we consider the case of genus two is not because this case is particularly distinguished, but as a first step towards a general construction for arbitrary genus.

\pagebreak

\section{Genus two Macdonald polynomials}

We fix the ground field ${\mathbf{k}} = {\mathbb C}\left( q^{\frac{1}{4}}, t^{\frac{1}{4}}\right)$ to be the field of rational functions in $q^{\frac{1}{4}}$ and $t^{\frac{1}{4}}$.

\begin{definition}
We call a triple of integers $(j_1,j_2,j_3)$ admissible if each of them is non-negative, $|j_1 - j_2| \leq j_3 \leq j_1 + j_2$ and $j_1 + j_2 + j_3$ is even.
\end{definition}

\begin{definition}
Consider a family of Laurent polynomials $\Psi_{j_1,j_2,j_3}$ in variables $x_{12},x_{13},x_{23}$ which are a common solution to the following system of recursive relations, satisfied for all admissible triples $(j_1,j_2,j_3)$:

\begin{subequations}
\begin{align}
\label{Pieri12}(x_{12} + x_{12}^{-1}) \ \Psi_{j_1,j_2,j_3} = \sum\limits_{a,b \in \{-1,+1\}} \ C_{a,b}(j_1,j_2,j_3) \ \Psi_{j_1+a,j_2+b,j_3} \\
\nonumber \\
\label{Pieri13}(x_{13} + x_{13}^{-1}) \ \Psi_{j_1,j_2,j_3} = \sum\limits_{a,b \in \{-1,+1\}} \ C_{a,b}(j_1,j_3,j_2) \ \Psi_{j_1+a,j_2,j_3+b} \\
\nonumber \\
\label{Pieri23}(x_{23} + x_{23}^{-1}) \ \Psi_{j_1,j_2,j_3} = \sum\limits_{a,b \in \{-1,+1\}} \ C_{a,b}(j_2,j_3,j_1) \ \Psi_{j_1,j_2+a,j_3+b}
\end{align}
\label{Pieri}
\end{subequations}
\smallskip\\
with initial conditions: $\Psi_{j_1,j_2,j_3} = 0$ for non-admissible triples $(j_1,j_2,j_3)$, and $\Psi_{0,0,0} = 1$. Here

\begin{align}
C_{a,b}(j_1,j_2,j_3) = a \ b \ \frac{
\left[ \frac{a j_1 + b j_2 + j_3}{2}, \frac{a + b + 2}{2} \right]_{q,t}
\left[ \frac{a j_1 + b j_2 - j_3}{2}, \frac{a + b}{2} \right]_{q,t}
\Big[j_1 - 1, 2\Big]_{q,t}
\Big[j_2 - 1, 2\Big]_{q,t}
}
{
\Big[j_1, \frac{a+3}{2} \Big]_{q,t}\Big[j_1 - 1, \frac{a+3}{2}\Big]_{q,t}\Big[j_2, \frac{a+3}{2} \Big]_{q,t}\Big[j_2 - 1, \frac{a+3}{2}\Big]_{q,t}
}
\label{Cab}
\end{align}
\smallskip\\
where
\begin{align*}
[n,m]_{q,t} = \dfrac{ q^{\frac{n}{2}} t^{\frac{m}{2}} - q^{-\frac{n}{2}} t^{-\frac{m}{2}} }{ q^{\frac{1}{2}} - q^{-\frac{1}{2}} }
\end{align*}
We call $\Psi_{j_1,j_2,j_3}$ genus two Macdonald polynomials (of type $A_1$), and recursion (\ref{Pieri}) genus two Pieri rule. Their existence follows from the following Lemma and Proposition.

\end{definition}

\begin{lemma}
For a given admissible triple $(j_1,j_2,j_3)$, the coefficient $C_{a,b}(j_1,j_2,j_3)$ is non-vanishing if and only if the triple $(j_1+a,j_2+b,j_3)$ is admissible.
\end{lemma}
\begin{proof}
First note that $a + b = 0$ mod $2$, so that we don't have to worry about the parity condition. As for the other conditions, there are three cases to consider.
\begin{itemize}
\item{$(a,b) = (+1,+1)$} In this case $|j_1 - j_2|$ does not change, and $j_1 + j_2$ increases by 2, so the triple $(j_1+a,j_2+b,j_3)$ has to be admissible.
\item{$(a,b) = (\pm 1,\mp 1)$} In this case $j_1 + j_2$ does not change, but $|j_1 - j_2|$ changes by 2. Therefore the triple $(j_1+a,j_2+b,j_3)$ is admissible unless $a j_1 + b j_2 = j_3$. This is precisely when the second factor in the numerator of (\ref{Cab}) vanishes.
\item{$(a,b) = (-1,-1)$} In this case $|j_1 - j_2|$ does not change, and $j_1 + j_2$ decreases by 2. Therefore the triple $(j_1+a,j_2+b,j_3)$ is admissible unless $j_1 + j_2 = j_3$. This is precisely when the first factor in the numerator of (\ref{Cab}) vanishes.
\end{itemize}
\end{proof}
This technical lemma is important: since each triple that appears on the r.h.s. of (\ref{Pieri}) with a non-vanishing coefficient is necessarily admissible, one can iterate the genus two Pieri rule.

\pagebreak

\begin{proposition}
Recursive relations (\ref{Pieri12}), (\ref{Pieri13}), and (\ref{Pieri23}) are compatible.
\end{proposition}
\begin{proof}
Let us first prove compatibility of (\ref{Pieri12}) and (\ref{Pieri13}). Applying (\ref{Pieri12}) followed by (\ref{Pieri13}), we get

\begin{align*}
(x_{13} + x_{13}^{-1}) (x_{12} + x_{12}^{-1}) \Psi_{j_1,j_2,j_3} = \sum\limits_{a_1,b_1} \sum\limits_{a_2,b_2} \ C_{a_1, b_1}(j_1,j_2,j_3) \ C_{a_2, b_2}(j_1+a_1,j_3,j_2+b_1) \ \Psi_{j_1+a_1+a_2,j_2+b_1,j_3+b_2}
\end{align*}
\smallskip\\
In the opposite order we get

\begin{align*}
(x_{12} + x_{12}^{-1}) (x_{13} + x_{13}^{-1}) \Psi_{j_1,j_2,j_3} = \sum\limits_{a_1,b_1} \sum\limits_{a_2,b_2} \ C_{a_1, b_2}(j_1,j_3,j_2) \ C_{a_2, b_1}(j_1+a_1,j_2,j_3+b_2) \ \Psi_{j_1+a_1+a_2,j_2+b_1,j_3+b_2}
\end{align*}
\smallskip\\
Compatibility of (\ref{Pieri12}) and (\ref{Pieri13}) follows from the vanishing of their difference

\begin{align*}
\sum\limits_{a_1 + a_2 = a} \big( C_{a_1, b_1}(j_1,j_2,j_3) \ C_{a_2, b_2}(j_1+a_1,j_3,j_2+b_1) -  C_{a_1, b_2}(j_1,j_3,j_2) \ C_{a_2, b_1}(j_1+a_1,j_2,j_3+b_2) \big) = 0
\end{align*}
\smallskip\\
for all $a \in \{-2,0,2\}$, $b_1, b_2 \in \{-1,1\}$, and $j_1, j_2, j_3$, which is a direct computation in ${\mathbf{k}} = {\mathbb C}\left( q^{\frac{1}{4}}, t^{\frac{1}{4}}\right)$. Compatibility of the other two pairs follows by permutation of $(j_1,j_2,j_3)$.

\end{proof}

\begin{corollary}
Genus two Macdonald polynomials are symmetric w.r.t. sumultaneous permutations:
\begin{align}
\Psi_{j_1,j_2,j_3}(x_{12},x_{13},x_{23}) = \Psi_{j_{\sigma_1},j_{\sigma_2},j_{\sigma_3}}(x_{\sigma_1 \sigma_2},x_{\sigma_1 \sigma_3},x_{\sigma_2 \sigma_3}), \ \ \ \forall \sigma \in S_3
\label{eq:PsiS3Symmetry}
\end{align}
\label{Symmetry}
\hspace{1ex}
\end{corollary}
\begin{corollary}
Usual Macdonald polynomials \cite{Macdonald} are a particular case of genus two Macdonald polynomials,
\begin{align*}
\Psi_{\ell, \ell, 0}(x_{12},x_{13},x_{23}) = c_{\ell} P_{\ell}(x_{12}), \ \ \ \Psi_{\ell, 0, \ell}(x_{12},x_{13},x_{23}) = c_{\ell} P_{\ell}(x_{13}), \ \ \ \Psi_{0, \ell, \ell}(x_{12},x_{13},x_{23}) = c_{\ell} P_{\ell}(x_{23})
\end{align*}
where $c_\ell$ is a normalization constant, which is nothing but the well-known principal specialization of $P_{\ell}$:
\begin{align*}
c_\ell := t^{\frac{\ell}{2}} \ \prod\limits_{i = 0}^{\ell - 1} \dfrac{1 - q^i t}{1 - q^i t^2} = P_\ell(t^{\frac{1}{2}})
\end{align*}
\label{Reduction}
\end{corollary}
\begin{proof}
By corollary \ref{Symmetry}, it suffices to prove one of the three specializations, the other two will follow. Consider the case $j_3 = 0$ of the genus two Pieri rule (\ref{Pieri12}). It is easy to see that it takes form
\begin{align*}
(x_{12} + x_{12}^{-1})\Psi_{\ell, \ell, 0} = t^{\frac{1}{2}} \ \dfrac{1 - q^\ell t}{1 - q^\ell t^2} \ \Psi_{\ell+1, \ell+1, 0} + t^{-\frac{1}{2}} \ \dfrac{(1 - q^\ell)(1 - q^{\ell - 1} t^2)^2}{(1 - q^\ell t)(1 - q^{\ell-1} t)^2} \ \Psi_{\ell-1, \ell-1, 0}
\end{align*}
After renormalization $\Psi_{\ell, \ell, 0} =  c_{\ell} \Psi^{\prime}_{\ell, \ell, 0}$ this recursion takes form
\begin{align*}
(x_{12} + x_{12}^{-1})\Psi^{\prime}_{\ell, \ell, 0} = \Psi^{\prime}_{\ell+1, \ell+1, 0} + \dfrac{(1 - q^\ell)(1 - q^{\ell - 1} t^2)}{(1 - q^\ell t)(1 - q^{\ell-1} t)} \ \Psi^{\prime}_{\ell-1, \ell-1, 0}
\end{align*}
This is precisely the Pieri rule for usual Macdonald polynomials, with the same initial conditions $\Psi^{\prime}_{\ell, \ell, 0} = 0$ for $\ell < 0$ and $\Psi^{\prime}_{0, 0, 0} = 1$. Consequently, $\Psi^{\prime}_{\ell, \ell, 0} = P_{\ell}(x_{12})$.
\end{proof}

\begin{lemma}
Genus two Macdonald polynomials $\Psi_{j_1,j_2,j_3}$ have a form
\label{lemm:PsiLeadingTermExpansion}
\begin{align}
\Psi_{j_1,j_2,j_3} = x_{12}^{d_3}x_{13}^{d_2}x_{23}^{d_1} \ \mathop{\sum\limits_{n_1,n_2,n_3 \geq 0}}_{n_1 + n_2 + n_3 \leq j_1 + j_2 + j_3} \ K_{n_1,n_2,n_3}(j_1,j_2,j_3) \ \left( \dfrac{x_{23}}{x_{12}x_{13}} \right)^{n_1} \ \left( \dfrac{x_{13}}{x_{12}x_{23}} \right)^{n_2} \ \left( \dfrac{x_{12}}{x_{13}x_{23}} \right)^{n_3}
\label{Ansatz}
\end{align}
where
\begin{align*}
d_{1} = \dfrac{-j_1+j_2+j_3}{2}, \ \ \ d_{2} = \dfrac{j_1-j_2+j_3}{2}, \ \ \ d_{3} = \dfrac{j_1+j_2-j_3}{2}
\end{align*}
and $K_{0,0,0}(j_1,j_2,j_3) \neq 0$ for all admissible triples $(j_1,j_2,j_3)$.
\end{lemma}
\begin{proof}
The proof is by induction in $j_1+j_2+j_3$. The base case $(j_1,j_2,j_3) = (0,0,0)$ is obvious. Assume that the lemma holds for all admissible triples $(j_1,j_2,j_3)$ such that $j_1+j_2+j_3 \leq J$. Pick any admissible triple with $(j_1+j_2+j_3) = J$. Applying the Pieri rule (\ref{Pieri12}) we get $\Psi_{j_1+1,j_2+1,j_3}$ as a ${\mathbf{k}}$-linear combination of $\Psi_{j_1-1,j_2+1,j_3}, \Psi_{j_1+1,j_2-1,j_3}, \Psi_{j_1-1,j_2-1,j_3}$ and $(x_{12} + x_{12}^{-1})\Psi_{j_1,j_2,j_3}$. Consider first $\Psi_{j_1-1,j_2+1,j_3}$. By the inductive assumption, it has the form (\ref{Ansatz}) with the leading term

\begin{align*}
x_{12}^{d_3+2}x_{13}^{d_2-2}x_{23}^{d_1-2} = x_{12}^{d_3}x_{13}^{d_2}x_{23}^{d_1} \cdot \left(\dfrac{x_{12}}{x_{13}x_{23}}\right)^2
\end{align*}
\smallskip\\
therefore, every monomial in the expansion of $\Psi_{j_1-1,j_2+1,j_3}$ has a form

\begin{align}
x_{12}^{d_3}x_{13}^{d_2}x_{23}^{d_1} \cdot \ \left( \dfrac{x_{23}}{x_{12}x_{13}} \right)^{n_1} \ \left( \dfrac{x_{13}}{x_{12}x_{23}} \right)^{n_2} \ \left( \dfrac{x_{12}}{x_{13}x_{23}} \right)^{n_3}
\label{formform}
\end{align}
\smallskip\\
where the sum $n_1+n_2+n_3$ cannot be bigger than $j_1+j_2+j_3+2$. The same holds for $\Psi_{j_1+1,j_2-1,j_3}$ and $\Psi_{j_1-1,j_2-1,j_3}$. As for $\Psi_{j_1,j_2,j_3}$, by the inductive assumption, it has the form (\ref{Ansatz}) with the leading term

\begin{align*}
x_{12}^{d_3}x_{13}^{d_2}x_{23}^{d_1} \cdot \ x_{12}^{-1}
\end{align*}
\smallskip\\
therefore every monomial in the expansion of $(x_{12} + x_{12}^{-1}) \Psi_{j_1,j_2,j_3}$ has the form (\ref{formform}) where the sum $n_1+n_2+n_3$ cannot be bigger than $j_1+j_2+j_3+2$, and the coefficient in front of $x_{12}^{d_3}x_{13}^{d_2}x_{23}^{d_1}$ is non-vanishing. We therefore proved that $\Psi_{j_1+1,j_2+1,j_3}$ has this form too.
\end{proof}

Let ${\cal H}$ be the space of all Laurent polynomials in $x_{12},x_{13},x_{23}$ (as everywhere in this paper, with coefficients over ${\mathbf{k}}$) symmetric under the ${\mathbb Z}_2^3$ group of Weyl inversions

\begin{align*}
(x_{12},x_{13},x_{23}) \ \ \ \mapsto \ \ \ (x_{12}^u,x_{13}^v,x_{23}^w), \ \ \ \ \ \ u,v,w \in \{ -1, +1 \}
\end{align*}

\begin{proposition}
Genus two Macdonald polynomials $\Psi_{j_1,j_2,j_3}$ for all admissible $(j_1,j_2,j_3)$ form a basis of ${\cal H}$.
\end{proposition}
\begin{proof}
Let us first prove that ${\cal H}$ is spanned by genus two Macdonald polynomials. For this, note that ${\cal H}$ is spanned by the set of monomials $\{ \ p_{12}^{m_{12}} p_{13}^{m_{13}} p_{23}^{m_{23}} | \ m_{12}, m_{13}, m_{23} \geq 0 \ \}$ where $p_{ab} = x_{ab} + x_{ab}^{-1}$. This set contains the generators $p_{12} \propto \Psi_{1,1,0}, \ p_{13} \propto \Psi_{1,0,1}, \ p_{23} \propto \Psi_{0,1,1}$. Since span($\Psi_{j_1,j_2,j_3}$) is closed under multiplication by $p_{12},p_{13},p_{23}$ (see (\ref{Pieri})) genus two Macdonald polynomials span ${\cal H}$. By Lemma \ref{lemm:PsiLeadingTermExpansion} their leading terms in total degree are all different, thus they are linearly independent and form a basis of $\mathcal H$.
\end{proof}

\pagebreak

\section{Algebra of Knot operators}

Denote by $\mathcal F:=\mathbf{k}(x_{12},x_{23},x_{13})\big[\hat\delta_{12},\hat\delta_{23}, \hat\delta_{13}\big]$ the algebra of all $q$-difference operators with coefficients to be rational functions in $x_{12},x_{23},x_{13}$ over $\mathbf{k}$. Here
\begin{align*}
\hat\delta_{12}f(x_{12},x_{23},x_{13}) &=f\left(q^{\frac12}x_{12},x_{23},x_{13}\right)\quad\textrm{for all}\quad f\in\mathbf{k}(x_{12},x_{23},x_{13})
\end{align*}
and the other two: $\hat\delta_{23}$ and $\hat\delta_{13}$ are obtained by permutation of indexes.

\begin{definition}
Let ${\hat {\cal O}}_{A_1}, {\hat {\cal O}}_{A_2}, {\hat {\cal O}}_{A_3}$ be the following $q$-difference operators acting on ${\mathbf{k}}(x_{12},x_{13},x_{23})$

\begin{subequations}
\begin{align}
\label{Hamiltonian1}{\hat {\cal O}}_{A_1} \ = \ \sum\limits_{a,b \in \{\pm 1\}} \ a b \ \dfrac{(1 - t^{\frac{1}{2}} x_{23}^{+1} x_{12}^a x_{13}^b)(1 - t^{\frac{1}{2}} x_{23}^{-1} x_{12}^a x_{13}^b)}{t^{\frac{1}{2}} x_{12}^{a} x_{13}^b (x_{12}^{+1} - x_{12}^{-1})(x_{13}^{+1} - x_{13}^{-1})} \ {\hat \delta}_{12}^{a} {\hat \delta}_{13}^{b} \\
\nonumber \\
\label{Hamiltonian2}{\hat {\cal O}}_{A_2} \ = \ \sum\limits_{a,b \in \{\pm 1\}} \ a b \ \dfrac{(1 - t^{\frac{1}{2}} x_{13}^{+1} x_{12}^a x_{23}^b)(1 - t^{\frac{1}{2}} x_{13}^{-1} x_{12}^a x_{23}^b)}{t^{\frac{1}{2}} x_{12}^{a} x_{23}^b (x_{12}^{+1} - x_{12}^{-1})(x_{23}^{+1} - x_{23}^{-1})} \ {\hat \delta}_{12}^{a} {\hat \delta}_{23}^{b} \\
\nonumber \\
\label{Hamiltonian3}{\hat {\cal O}}_{A_3} \ = \ \sum\limits_{a,b \in \{\pm 1\}} \ a b \ \dfrac{(1 - t^{\frac{1}{2}} x_{12}^{+1} x_{13}^a x_{23}^b)(1 - t^{\frac{1}{2}} x_{12}^{-1} x_{13}^a x_{23}^b)}{t^{\frac{1}{2}} x_{13}^{a} x_{23}^b (x_{13}^{+1} - x_{13}^{-1})(x_{23}^{+1} - x_{23}^{-1})} \ {\hat \delta}_{13}^{a} {\hat \delta}_{23}^{b}
\end{align}
\label{Hamiltonian}
\end{subequations}
\smallskip\\
and ${\hat {\cal O}}_{B_{12}}, {\hat {\cal O}}_{B_{13}}, {\hat {\cal O}}_{B_{23}}$ be the following multiplication operators:

\begin{subequations}
\begin{align}
\label{Mult1}{\hat {\cal O}}_{B_{12}} \ = \ x_{12} + x_{12}^{-1} \\[10pt]
\label{Mult2}{\hat {\cal O}}_{B_{13}} \ = \ x_{13} + x_{13}^{-1} \\[10pt]
\label{Mult3}{\hat {\cal O}}_{B_{23}} \ = \ x_{23} + x_{23}^{-1}
\end{align}
\label{Mult}
\end{subequations}
\end{definition}

\begin{definition}
Let $\mathcal A:=\mathbf{k}[\hat{\mathcal O}_{A_1},\hat{\mathcal O}_{A_2},\hat{\mathcal O}_{A_3},\hat{\mathcal O}_{B_{12}},\hat{\mathcal O}_{B_{23}},\hat{\mathcal O}_{B_{13}}] \subset \mathcal F$ be the algebra generated by the above six operators. We will call $\mathcal A$ the algebra of genus two knot operators.
\end{definition}

\begin{remark}
The action of the operator ${\hat {\cal O}}_{A_1}$ on functions that depend only on $x_{12}$ is identical to the action of the Macdonald operator ${\hat {\cal O}}_{A}$ of the spherical DAHA (\ref{OAOB}). Therefore the action of the subalgebra $\mathbf{k}[\hat{\mathcal O}_{A_1},\hat{\mathcal O}_{B_{12}}] \subset \mathcal A$ on functions that depend only on $x_{12}$ is isomorphic to the action of the $A_1$ spherical DAHA. By $S_3$ symmetry, a similar statement holds for $x_{13}$ and $x_{23}$.
\end{remark}

\begin{lemma}
The algebra ${\cal A}$ is not free: among others, the following relations hold for all $1\leq i<j\leq3$

\begin{subequations}
\begin{align}
[ {\hat{\cal O}}_{A_1}, {\hat{\cal O}}_{A_2} ]= [ {\hat{\cal O}}_{A_2}, {\hat{\cal O}}_{A_3} ]=[ {\hat{\cal O}}_{A_1}, {\hat{\cal O}}_{A_3} ]=0
\end{align}
\begin{align}
[\hat{\mathcal O}_{B_{12}},\hat{\mathcal O}_{B_{23}}]=[\hat{\mathcal O}_{B_{12}},\hat{\mathcal O}_{B_{13}}]=[\hat{\mathcal O}_{B_{23}},\hat{\mathcal O}_{B_{13}}] = 0
\label{eq:KnotAlgebraCommutativity1}
\end{align}
\begin{align}
[\hat{\mathcal O}_{A_1},\hat{\mathcal O}_{B_{23}}]=[\hat{\mathcal O}_{A_2},\hat{\mathcal O}_{B_{13}}]=[\hat{\mathcal O}_{A_3},\hat{\mathcal O}_{B_{12}}]=0
\label{eq:KnotAlgebraCommutativity2}
\end{align}
\begin{align}
& \big(q^{\frac{1}{2}} - q^{-\frac{1}{2}}\big)^2 \ {\hat{\cal O}}_{B_{ij}} = -{\hat{\cal O}}_{A_i} {\hat{\cal O}}_{A_i} {\hat{\cal O}}_{B_{ij}} + \big(q^{\frac{1}{2}} + q^{-\frac{1}{2}}\big) \ {\hat{\cal O}}_{A_i} {\hat{\cal O}}_{B_{ij}} {\hat{\cal O}}_{A_i} - {\hat{\cal O}}_{B_{ij}} {\hat{\cal O}}_{A_i} {\hat{\cal O}}_{A_i}
\label{eq:KnotAlgebraQSerreABA} \\
\nonumber \\
& \big(q^{\frac{1}{2}} - q^{\frac{-1}{2}}\big)^2 \ {\hat{\cal O}}_{A_i} = -{\hat{\cal O}}_{B_{ij}} {\hat{\cal O}}_{B_{ij}} {\hat{\cal O}}_{A_i} + \big(q^{\frac{1}{2}} + q^{\frac{-1}{2}}\big) \ {\hat{\cal O}}_{B_{ij}} {\hat{\cal O}}_{A_i} {\hat{\cal O}}_{B_{ij}} - {\hat{\cal O}}_{A_i} {\hat{\cal O}}_{B_{ij}} {\hat{\cal O}}_{B_{ij}}
\label{eq:KnotAlgebraQSerreBAB}
\end{align}
\begin{align}
\hat{\mathcal O}_{B_{13}}\hat{\mathcal O}_{A_1}\hat{\mathcal O}_{B_{12}}-\hat{\mathcal O}_{B_{12}}\hat{\mathcal O}_{A_1}\hat{\mathcal O}_{B_{13}}=\hat{\mathcal O}_{A_2}\hat{\mathcal O}_{B_{23}}\hat{\mathcal O}_{A_3}-\hat{\mathcal O}_{A_3}\hat{\mathcal O}_{B_{23}}\hat{\mathcal O}_{A_2}
\label{eq:KnotAlgebraExtra1}
\end{align}
\begin{align}
\hat{\mathcal O}_{A_1}\hat{\mathcal O}_{B_{12}}\hat{\mathcal O}_{B_{13}}-&\big(q^{\frac12}+q^{-\frac12}\big)\hat{\mathcal O}_{B_{12}}\hat{\mathcal O}_{A_1}\hat{\mathcal O}_{B_{13}}+\hat{\mathcal O}_{B_{13}}\hat{\mathcal O}_{B_{12}}\hat{\mathcal O}_{A_1}=\nonumber\\[5pt]
&\hat{\mathcal O}_{A_3}\hat{\mathcal O}_{A_2}\hat{\mathcal O}_{B_{23}}-\big(q^{\frac12}+q^{-\frac12}\big)\hat{\mathcal O}_{A_3}\hat{\mathcal O}_{B_{23}}\hat{\mathcal O}_{A_2}+\hat{\mathcal O}_{B_{23}}\hat{\mathcal O}_{A_2}\hat{\mathcal O}_{A_3}
\label{eq:KnotAlgebraExtra2}
\end{align}
\label{eq:KnotAlgebraRelations}
\end{subequations}
\end{lemma}
\begin{proof}
For each relation apply both sides to an arbitrary $f\in\mathbf k(x_{12},x_{23},x_{13}),$ the result will follow by a direct computation in $\mathbf k(x_{12},x_{23},x_{13})$.
\end{proof}

\begin{proposition}
Genus two Macdonald polynomials are common eigenfunctions of ${\hat {\cal O}}_{A_1}, {\hat {\cal O}}_{A_2}, {\hat {\cal O}}_{A_3}$
\begin{align}
\hat{\mathcal O}_{A_k}\Psi_{j_1,j_2,j_3}=\left(q^{\frac {j_k}2}t^{\frac12}+q^{-\frac {j_k}2}t^{-\frac12}\right)\Psi_{j_1,j_2,j_3}
\label{eq:EigenvaluesOA}
\end{align}
\label{propeigen}
\end{proposition}
\begin{proof}
See Appendix \ref{AppA}.
\end{proof}

\begin{corollary}
Algebra of knot operators is acting on the space of genus two symmetric polynomials $\mathcal H$.
\end{corollary}
\begin{proof}
Indeed, $\mathcal A$ is acting on $\mathbf k\left(x_{12},x_{23},x_{13}\right)$, so the only thing we have to prove is that this action preserves $\mathcal H\subset\mathbf k(x_{12},x_{23},x_{13})$. Combining (\ref{Mult}) with (\ref{eq:EigenvaluesOA}) we conclude that generators of $\mathcal A$ leave $\mathcal H$ invariant and thus so does $\mathcal A$.
\end{proof}

This allows one to define a polynomial representation $\mathcal{A_H}\subset End(\mathcal H)$ of algebra $\mathcal A$. In what follows we denote the natural image of generators $\hat{\mathcal O}_{B_{ij}},\hat{\mathcal O}_{A_k}\subset \mathcal A_H$ by the same letters.

\section{Mapping Class Group Action}

\begin{definition}
Define $b_{12},b_{23},b_{13}$ to be the following automorphisms of $\mathcal F=\mathbf k(x_{12},x_{23},x_{13})[\hat\delta_{12},\hat\delta_{23},\hat\delta_{13}]$
\begin{align}
b_{ij}\big(x_{kl}\big)=x_{kl},
\qquad\qquad
b_{ij}\big(\hat\delta_{kl}\big)=\left\{\begin{array}{cll}
q^{\frac14}x_{ij}\,\hat\delta_{ij}&&(ij)=(kl),\\[10pt]
\hat\delta_{kl}&&(ij)\neq(kl)
\end{array}\right.
\qquad
\textrm{for all}\quad1\leq k<l\leq 3.
\label{eq:BAutDef}
\end{align}
\end{definition}
Note that all relations in $\mathcal F$ are preserved by (\ref{eq:BAutDef}).

\begin{lemma}
Automorphisms $b_{12}, b_{23}, b_{13}$ act on the generators of $\mathcal A$ in the following way
\begin{subequations}
\begin{align}
b_{ij}^{\pm1}\big(\hat{\mathcal O}_{A_k}\big)=&\left\{\begin{array}{lll}
\pm\dfrac{q^{\pm\frac14}}{q^{\frac12}-q^{\frac12}}\hat{\mathcal O}_{A_k}\hat{\mathcal O}_{B_{ij}}\mp\dfrac{q^{\mp\frac14}}{q^{\frac12}-q^{-\frac12}}\hat{\mathcal O}_{B_{ij}}\hat{\mathcal O}_{A_k}&&k\in\{i,j\},\\[15pt]
\hat{\mathcal O}_{A_k}&&k\not\in\{i,j\},
\end{array}\right.
\label{eq:BAutActionOA}
\\\nonumber\\
b_{ij}^{\pm1}\big(\hat{\mathcal O}_{B_{kl}}\big)=&\hat{\mathcal O}_{B_{kl}}\qquad\textrm{for all}\quad 1\leq k<l\leq3.
\label{eq:BAutActionOB}
\end{align}
\label{eq:BAutAction}
\end{subequations}
\end{lemma}
\begin{proof}
Identity (\ref{eq:BAutActionOB}) is immediate. To prove (\ref{eq:BAutActionOA}) note that (\ref{eq:BAutDef}) imply $b_{ij}\big(\hat\delta_{ij}^m\big)=q^{m^2/4}x_{ij}^m\hat\delta_{ij}^m$. Applying $b_{ij}$ to (\ref{Hamiltonian1})--(\ref{Hamiltonian3}) and collecting coefficients in front of $\hat\delta_{ij}^{\pm1}\hat\delta_{ik}^{\pm1}$ on both sides of (\ref{eq:BAutActionOA}) we get the result.
\end{proof}

\begin{corollary}
Algebra $\mathcal A$ is invariant under $b_{12}^{\pm1}, b_{23}^{\pm1}, b_{13}^{\pm1}$, as a result the above automorphisms act on $\mathcal A_H\subset\mathrm{End}(\mathcal H)$ as well.
\label{cor:BTwistActionAH}
\end{corollary}

\begin{definition}
Let $A_1,A_2,A_3\in\mathrm{End}(\mathcal H)$ be defined on basis elements as
\begin{align}
A_k\psi_{j_1,j_2,j_3}=q^{\frac{j_k^2}4}t^{\frac{j_k}2}\psi_{j_1,j_2,j_3}.
\label{eq:AEndAction}
\end{align}
Denote the corresponding adjoint action of $A_1,A_2,A_3$ on $\mathrm{End}(\mathcal H)$ as
\begin{align*}
a_k(h)=A_k^{-1}\,h\,A_k,\qquad\textrm{for all}\quad h\in\mathcal H.
\end{align*}
\end{definition}

\begin{lemma}
Automorphisms $a_1,a_2,a_3$ act on generators of $\mathcal A_H$ in the following way
\begin{subequations}
\begin{align}
a_k^{\pm1}\big(\hat{\mathcal O}_{B_{ij}}\big)=&\left\{\begin{array}{lcc}
\mp\dfrac{q^{\mp\frac14}}{q^{\frac12}-q^{-\frac12}}\hat{\mathcal O}_{A_k} \hat{\mathcal O}_{B_{ij}}\pm \dfrac{q^{\pm\frac14}}{q^{\frac12}-q^{-\frac12}} \hat{\mathcal O}_{B_{ij}}\hat{\mathcal O}_{A_k}&& k\in\{i,j\},\\[15pt]
\hat{\mathcal O}_{B_{ij}}&&k\not\in\{i,j\},
\end{array}\right.
\label{eq:AAutActionOB}
\\\nonumber\\
a_k^{\pm}\big(\hat{\mathcal O}_{A_j}\big)=&\;\hat{\mathcal O}_{A_j}\qquad\textrm{for all}\quad 1\leq j,k\leq 3.
\label{eq:AAutActionOA}
\end{align}
\label{eq:AAutAction}
\end{subequations}
\label{lemm:AAutAction}
\end{lemma}
\begin{proof}
Identity (\ref{eq:AAutActionOA}) is immediate. To prove (\ref{eq:AAutActionOB}) apply both sides to an arbitrary basis element $\psi_{j_1,j_2,j_3}\subset\mathcal H$ and use (\ref{eq:EigenvaluesOA}) and (\ref{Pieri}) to calculate the r.h.s., whereas (\ref{eq:EigenvaluesOA}) and (\ref{eq:AEndAction}) for the l.h.s. Since the two agree we conclude that the identity holds between the operators in $\mathrm{End}(\mathcal H)$.
\end{proof}

Combining the above Lemma with Corollary \ref{cor:BTwistActionAH} we conclude that the group $G=\langle a_1,a_2,a_3,b_{12},b_{23},b_{13}\rangle$ acts by automorphisms of $\mathcal A_H$.

\begin{proposition}
Consider the following two automorphisms
\begin{align}
I:=a_1\circ b_{12}\circ a_2\circ b_{23}\circ a_3\qquad\widetilde I:=a_3\circ b_{23}\circ a_2\circ b_{12}\circ a_1.
\label{eq:GenusTwoFourierDef}
\end{align}
Then for all $1\leq i\leq 3$ we have\footnote{Here we assume $\hat{\mathcal O}_{A_4}:=\hat{\mathcal O}_{A_1}$ and $\hat{\mathcal O}_{B_{34}}:=\hat{\mathcal O}_{B_{13}}$.}
\begin{equation}
\begin{aligned}
&I(\hat{\mathcal O}_{A_i})=\hat{\mathcal O}_{B_{i,i+1}}&\qquad& \widetilde I(\hat{\mathcal O}_{B_{i,i+1}})=\hat{\mathcal O}_{A_i}\\
&I(\hat{\mathcal O}_{B_{i,i+1}})=\hat{\mathcal O}_{A_{i+1}}&\qquad& \widetilde I(\hat{\mathcal O}_{A_{i+1}})=\hat{\mathcal O}_{B_{i,i+1}}
\end{aligned}
\label{eq:GenusTwoFourierAct}
\end{equation}
\end{proposition}

\begin{proof}
First, note
\begin{align*}
(b_{12}\circ a_2)\big(\hat{\mathcal O}_{B_{12}}\big) =&\frac{-\hat{\mathcal O}_{A_2}\hat{\mathcal O}_{B_{12}}\hat{\mathcal O}_{B_{12}}+\big(q^{\frac12}+q^{-\frac12}\big)\hat{\mathcal O}_{B_{12}}\hat{\mathcal O}_{A_2}\hat{\mathcal O}_{B_{12}}-\hat{\mathcal O}_{B_{12}}\hat{\mathcal O}_{B_{12}}\hat{\mathcal O}_{A_2}}{\big(q^{\frac12}-q^{-\frac12}\big)^2}\quad\stackrel{\mathclap{\small \mbox{(\ref{eq:KnotAlgebraQSerreABA})}}}{=}\quad\hat{\mathcal O}_{A_2}\\[10pt]
(a_1\circ b_{12})\big(\hat{\mathcal O}_{A_1}\big)=&\frac{-\hat{\mathcal O}_{A_1}\hat{\mathcal O}_{A_1}\hat{\mathcal O}_{B_{12}}+\big(q^{\frac12}+q^{-\frac12}\big)\hat{\mathcal O}_{A_1}\hat{\mathcal O}_{B_{12}}\hat{\mathcal O}_{A_1}-\hat{\mathcal O}_{B_{12}}\hat{\mathcal O}_{A_1}\hat{\mathcal O}_{A_1}} {\big(q^{\frac12}-q^{-\frac12}\big)^2}\quad\stackrel{\mathclap{\small \mbox{(\ref{eq:KnotAlgebraQSerreBAB})}}}{=}\quad\hat{\mathcal O}_{B_{12}},
\end{align*}
by $S_3$ symmetry we get for all $1\leq i<j\leq3$
\begin{align}
(b_{ij}\circ a_j)\big(\hat{\mathcal O}_{B_{ij}}\big)=\hat{\mathcal O}_{A_j}\qquad
(a_i\circ b_{ij})\big(\hat{\mathcal O}_{A_i}\big)=\hat{\mathcal O}_{B_{ij}}.
\label{eq:EssentialFourier}
\end{align}
Now, combining (\ref{eq:BAutAction}), (\ref{eq:AAutAction}), and (\ref{eq:EssentialFourier}) we have
\begin{align*}
I(\hat{\mathcal O}_{B_{12}})=(a_1\circ b_{12}\circ a_2\circ b_{23}\circ a_3)\big(\hat{\mathcal O}_{B_{12}}\big)=(a_1\circ b_{12}\circ a_2)\big(\hat{\mathcal O}_{B_{12}}\big)=a_1\big(\hat{\mathcal O}_{A_1}\big) =\hat{\mathcal O}_{A_1}.
\end{align*}
By a similar computation we get
\begin{align*}
I(\hat{\mathcal O}_{A_1})=\hat{\mathcal O}_{B_{12}},\qquad I(\hat{\mathcal O}_{A_2})=\hat{\mathcal O}_{B_{23}},\qquad I(\hat{\mathcal O}_{B_{23}})=\hat{\mathcal O}_{A_{3}},
\end{align*}
\begin{align*}
\widetilde I(\hat{\mathcal O}_{A_2})=\hat{\mathcal O}_{B_{12}},\qquad
\widetilde I(\hat{\mathcal O}_{A_3})=\hat{\mathcal O}_{B_{23}},\qquad
\widetilde I(\hat{\mathcal O}_{B_{12}})=\hat{\mathcal O}_{A_1},\qquad \widetilde I(\hat{\mathcal O}_{B_{23}})=\hat{\mathcal O}_{A_2}.
\end{align*}
However in the case of the remaining four identities, the proof becomes slightly more delicate. Indeed,
\begin{align*}
(a_2\circ b_{23})\big(\hat{\mathcal O}_{A_3}\big)=&\;\frac{-\hat{\mathcal O}_{A_3}\hat{\mathcal O}_{A_2}\hat{\mathcal O}_{B_{23}}+q^{\frac12}\hat{\mathcal O}_{A_3}\hat{\mathcal O}_{B_{23}}\hat{\mathcal O}_{A_2}+q^{-\frac12}\hat{\mathcal O}_{A_2}\hat{\mathcal O}_{B_{23}}\hat{\mathcal O}_{A_3}-\hat{\mathcal O}_{B_{23}}\hat{\mathcal O}_{A_2}\hat{\mathcal O}_{A_3}}{\big(q^{\frac12}-q^{-\frac12}\big)^2}\\
\intertext{now, using (\ref{eq:KnotAlgebraExtra1}) and (\ref{eq:KnotAlgebraExtra2})}
=&\;\frac{-\hat{\mathcal O}_{A_1}\hat{\mathcal O}_{B_{12}}\hat{\mathcal O}_{B_{13}}+q^{\frac12}\hat{\mathcal O}_{B_{12}}\hat{\mathcal O}_{A_1}\hat{\mathcal O}_{B_{13}}+q^{-\frac12}\hat{\mathcal O}_{B_{13}}\hat{\mathcal O}_{A_1}\hat{\mathcal O}_{B_{12}}-\hat{\mathcal O}_{B_{13}}\hat{\mathcal O}_{B_{12}}\hat{\mathcal O}_{A_1}}{\big(q^{\frac12}-q^{-\frac12}\big)^2}\\[10pt]
=&\;(b_{1,2}^{-1}\circ a_1^{-1})\big(\hat{\mathcal O}_{B_{13}}\big).
\end{align*}
As a corollary
\begin{align}
I(\hat{\mathcal O}_{A_3})=(a_1\circ b_{12}\circ a_2\circ b_{23})\big(\hat{\mathcal O}_{A_3}\big)=\hat{\mathcal O}_{B_{13}}.
\label{eq:FourierHalfWay}
\end{align}
By a similar computation we get the remaining three identities
\begin{align*}
I(\hat{\mathcal O}_{B_{13}})=\hat{\mathcal O}_{A_1},\qquad \widetilde I(\hat{\mathcal O}_{A_1})=\hat{\mathcal O}_{B_{13}},\qquad \widetilde I(\hat{\mathcal O}_{B_{13}})=\hat{\mathcal O}_{A_3}.
\end{align*}
\end{proof}

\begin{theorem}
Mapping Class Group $\mathcal{MCG}(\Sigma_2)$ of a closed genus two surface is acting by outer \\
automorphisms of ${\mathcal A}_H$.
\end{theorem}
\begin{proof}
We will show that the outer automorphisms $a_1,a_2,a_3,b_{12},b_{23}$ satisfy the defining relations of the mapping class group of a closed genus two surface, as given in \cite{mcg}:
\begin{enumerate}
\begin{subequations}
\item Commutativity relations
\begin{align}
a_1\circ b_{23}=b_{23}\circ a_1,\qquad a_2\circ b_{13}=b_{13}\circ a_2,\qquad a_3\circ b_{12}=b_{12}\circ a_3.
\label{eq:MCGCommutativity}
\end{align}
\item Braid relations
\begin{align}
a_i\circ b_{ij}\circ a_i=b_{ij}\circ a_i\circ b_{ij}\qquad a_j\circ b_{ij}\circ a_j=b_{ij}\circ a_j\circ b_{ij}\quad\textrm{for all}\quad i<j.
\label{eq:MCGBraid}
\end{align}
\item Finite order relations
\begin{align}
I^6=1,\qquad H^2=1,\qquad H\circ a_i=a_i\circ H,\qquad H\circ b_{ij}=b_{ij}\circ H,
\label{eq:MCGFourier}
\end{align}
\begin{align}
(a_1\circ b_{12}\circ a_2)^4=a_3^2.
\label{eq:MCGLast}
\end{align}
\end{subequations}
where $I:=a_1\circ b_{12}\circ a_2\circ b_{23}\circ a_3$ and $H:=\widetilde I\circ I=a_3\circ b_{23}\circ a_2\circ b_{12}\circ a_1\circ a_1\circ b_{12}\circ a_2\circ b_{23}\circ a_3$
\end{enumerate}

To prove (\ref{eq:MCGCommutativity}) we compare action on generators $\hat{\mathcal O}_{A_1},\hat{\mathcal O}_{A_2},\hat{\mathcal O}_{A_3},\hat{\mathcal O}_{B_{12}},\hat{\mathcal O}_{B_{23}},\hat{\mathcal O}_{B_{13}}$ of ${\mathcal A}_H$ calculated in (\ref{eq:BAutAction}) and (\ref{eq:AAutAction}). The proof then follows by a straightforward computation in ${\mathcal A}_H$.

For (\ref{eq:MCGBraid}) it would be enough to prove that $a_1\circ b_{12}\circ a_2=b_{12}\circ a_1\circ b_{12},$ the others will follow by $S_3$ symmetry.
By (\ref{eq:BAutAction}) and (\ref{eq:AAutAction}) one immediately gets
\begin{subequations}
\begin{align}
(a_1\circ b_{12}\circ a_1)\big(\hat{\mathcal O}_{A_3}\big)=&\;\hat{\mathcal O}_{A_3}=(b_{12}\circ a_1\circ b_{12})\big(\hat{\mathcal O}_{A_3}\big),
\label{eq:ABAActionA3}\\[10pt]
(a_1\circ b_{12}\circ a_1)\big(\hat{\mathcal O}_{B_{23}}\big)=&\;\hat{\mathcal O}_{B_{23}}=(b_{12}\circ a_1\circ b_{12})\big(\hat{\mathcal O}_{B_{23}}\big).
\label{eq:ABAActionB23}
\end{align}
Using (\ref{eq:EssentialFourier}) we also get
\begin{align}
(a_1\circ b_{12}\circ a_1)\big(\hat{\mathcal O}_{A_1}\big)=&\;\hat{\mathcal O}_{B_{12}}=(b_{12}\circ a_1\circ b_{12})\big(\hat{\mathcal O}_{A_1}\big),
\label{eq:ABAActionA1}\\[10pt]
(a_1\circ b_{12}\circ a_1)\big(\hat{\mathcal O}_{B_{12}}\big)=&\;\hat{\mathcal O}_{A_1}=(b_{12}\circ a_1\circ b_{12})\big(\hat{\mathcal O}_{B_{12}}\big).
\label{eq:ABAActionB12}
\end{align}
\end{subequations}
Finally,
\begin{align*}
(a_1\circ b_{12}\circ a_1)\big(\hat{\mathcal O}_{A_2}\big)=&\; \frac{q^{\frac12}\hat{\mathcal O}_{A_2}\hat{\mathcal O}_{B_{12}}\hat{\mathcal O}_{A_1}+q^{-\frac12}\hat{\mathcal O}_{A_1}\hat{\mathcal O}_{B_{12}}\hat{\mathcal O}_{A_2}-\hat{\mathcal O}_{A_2}\hat{\mathcal O}_{A_1}\hat{\mathcal O}_{B_{12}}-\hat{\mathcal O}_{B_{12}}\hat{\mathcal O}_{A_1}\hat{\mathcal O}_{A_2}}{\big(q^{\frac12}-q^{-\frac12}\big)^2}\\[7pt]
=&\;(b_{12}\circ a_1\circ b_{12})\big(\hat{\mathcal O}_{A_2}\big),\\[20pt]
(a_1\circ b_{12}\circ a_1)\big(\hat{\mathcal O}_{B_{13}}\big)=&\; \frac{q^{\frac12}\hat{\mathcal O}_{B_{13}}\hat{\mathcal O}_{A_1}\hat{\mathcal O}_{B_{12}}+q^{-\frac12}\hat{\mathcal O}_{B_{12}}\hat{\mathcal O}_{A_1}\hat{\mathcal O}_{B_{13}}-\hat{\mathcal O}_{A_1}\hat{\mathcal O}_{B_{13}}\hat{\mathcal O}_{B_{12}}-\hat{\mathcal O}_{B_{12}}\hat{\mathcal O}_{B_{13}}\hat{\mathcal O}_{A_1}}{\big(q^{\frac12}+q^{-\frac12}\big)^2}\\[7pt]
=&\;(b_{12}\circ a_1\circ b_{12})\big(\hat{\mathcal O}_{B_{13}}\big).
\end{align*}

To prove (\ref{eq:MCGFourier}) recall that $I$ is acting on generators of $\mathcal A$ as a cyclic permutation of order 6 according to formula (\ref{eq:GenusTwoFourierAct}). Moreover, from (\ref{eq:GenusTwoFourierAct}) we have that $I$ and $\widetilde I$ are inverse to each other. This implies that $H=\widetilde I\circ I$ is a trivial automorphism of $\mathcal A$ i.e. $H=1$.

To prove the last relation (\ref{eq:MCGLast}) we also compare the action of both sides on generators. Immediately
\begin{subequations}
\begin{align}
(a_1\circ b_{12}\circ a_2)\big(\hat{\mathcal O}_{A_3}\big)=\hat{\mathcal O}_{A_3}
\label{eq:a1b12a2OrderOneOrbit}
\end{align}
Next, using (\ref{eq:EssentialFourier}) we have
\begin{align}
(a_1\circ b_{12}\circ a_2)\big(\hat{\mathcal O}_{A_1}\big)=\hat{\mathcal O}_{B_{12}},\qquad (a_1\circ b_{12}\circ a_2)\big(\hat{\mathcal O}_{B_{12}}\big)=\hat{\mathcal O}_{A_1}.
\label{eq:a1b12a2OrderTwoOrbit}
\end{align}
A bit more delicate computation shows
\begin{equation}
\begin{aligned}
(a_1\circ b_{12}\circ a_2)^2\big(\hat{\mathcal O}_{A_2})\;=&\quad((a_1\circ b_{12}\circ a_1)\circ (a_2\circ b_{12}\circ a_2))\big(\hat{\mathcal O}_{A_2}\big)\\[5pt]
\stackrel{\mathclap{\small \mbox{(\ref{eq:ABAActionA1})}}}{=}&\quad(a_1\circ b_{12}\circ a_1)\big(\hat{\mathcal O}_{B_{12}}\big)\\[5pt]
\stackrel{\mathclap{\small\mbox{(\ref{eq:ABAActionB12})}}}{=}&\quad \hat{\mathcal O}_{A_1}.
\end{aligned}
\label{eq:a1b12a2SquaredA2}
\end{equation}
Similarly to (\ref{eq:a1b12a2SquaredA2}) we have
\begin{align}
(a_1\circ b_{12}\circ a_2)^2\big(\hat{\mathcal O}_{A_1}\big)=\hat{\mathcal O}_{A_2}.
\label{eq:a1b12a2SquaredA1}
\end{align}
\label{eq:a1b12a2Easy}
\end{subequations}
Combining all equations (\ref{eq:a1b12a2Easy}) we conclude that $(a_1\circ b_{12}\circ a_2)^4$ acts trivially on $\hat{\mathcal O}_{A_1},\hat{\mathcal O}_{A_2},\hat{\mathcal O}_{A_3},$ and $\hat{\mathcal O}_{B_{12}}$ which coincides with an action of $a_3^2$.

\pagebreak

Equivalence of the action on $\hat{\mathcal O}_{B_{23}}$ follows from the following computation. First, we bring the identity to the equivalent form using commutativity relations (\ref{eq:MCGCommutativity}) and braid relations (\ref{eq:MCGBraid}) along with the fact that $a_1^{\pm1},b_{12}^{\pm1},b_{23}^{\pm1},$ and $b_{13}^{\pm}$ act trivially on $\hat{\mathcal O}_{B_{23}}$

\begin{align*}
(a_1\circ b_{12}\circ a_2\circ a_1\circ b_{12}\circ a_2\circ a_1\circ b_{12}\circ a_2\circ a_1\circ b_{12}\circ a_2)\big(\hat{\mathcal O}_{B_{23}}\big)=&\;a_3^2\big(\hat{\mathcal O}_{B_{23}}\big)\\
(a_2\circ a_1\circ b_{12}\circ a_2\circ a_1\circ b_{12}\circ a_1\circ a_2\circ b_{12}\circ a_2)\big(\hat{\mathcal O}_{B_{23}}\big)=&\;a_3^2\big(\hat{\mathcal O}_{B_{23}}\big)\\
(a_2\circ a_1\circ b_{12}\circ a_2\circ a_1\circ b_{12}\circ a_1\circ b_{12}\circ a_2\circ b_{12})\big(\hat{\mathcal O}_{B_{23}}\big)=&\;a_3^2\big(\hat{\mathcal O}_{B_{23}}\big)\\
(a_2\circ a_1\circ b_{12}\circ a_2\circ a_1\circ b_{12}\circ a_1\circ b_{12}\circ a_2)\big(\hat{\mathcal O}_{B_{23}}\big)=&\;a_3^2\big(\hat{\mathcal O}_{B_{23}}\big)\\
(a_2\circ a_1\circ b_{12}\circ a_2\circ a_1\circ a_1\circ b_{12}\circ a_1\circ a_2)\big(\hat{\mathcal O}_{B_{23}}\big)=&\;a_3^2\big(\hat{\mathcal O}_{B_{23}}\big)\\
(a_2\circ b_{12}\circ a_1\circ a_1\circ a_2\circ b_{12}\circ a_2)\big(\hat{\mathcal O}_{B_{23}}\big)=&\;a_3^2\big(\hat{\mathcal O}_{B_{23}}\big)\\
(a_2\circ b_{12}\circ a_1\circ a_1\circ b_{12}\circ a_2\circ b_{12})\big(\hat{\mathcal O}_{B_{23}}\big)=&\;a_3^2\big(\hat{\mathcal O}_{B_{23}}\big)\\
(a_2\circ b_{12}\circ a_1\circ a_1\circ b_{12}\circ a_2)\big(\hat{\mathcal O}_{B_{23}}\big)=&\;a_3^2\big(\hat{\mathcal O}_{B_{23}}\big)\\
(a_1\circ a_1\circ b_{12}\circ a_2)\big(\hat{\mathcal O}_{B_{23}}\big)=&\;(a_3^2\circ b_{12}^{-1}\circ a_2^{-1})\big(\hat{\mathcal O}_{B_{23}}\big)\\
\intertext{using (\ref{eq:FourierHalfWay}) to rewrite the r.h.s. we have}
(a_1\circ a_1\circ b_{12}\circ a_2)\big(\hat{\mathcal O}_{B_{23}}\big)=&\;(a_3^2\circ a_1\circ b_{13})\big(\hat{\mathcal O}_{A_3}\big)\\
(a_1\circ a_1\circ b_{12}\circ a_2)\big(\hat{\mathcal O}_{B_{23}}\big)=&\;(a_3\circ a_1)\big(\hat{\mathcal O}_{B_{13}}\big)\\
a_2\big(\hat{\mathcal O}_{B_{23}}\big)=&\;(a_3\circ b_{12}^{-1}\circ a_1^{-1})\big(\hat{\mathcal O}_{B_{13}}\big)\\
a_2\big(\hat{\mathcal O}_{B_{23}})=&\;(a_3\circ a_2\circ b_{23})\big(\hat{\mathcal O}_{A_3}\big)\\
\hat{\mathcal O}_{B_{23}}=&(a_3\circ b_{23})\big(\hat{\mathcal O}_{A_3}\big),
\end{align*}
\smallskip\\
where the latter is equivalent to (\ref{eq:EssentialFourier}).

By $S_3$ symmetry we also get

\begin{align*}
(a_1\circ b_{12}\circ a_2\circ a_2\circ b_{12}\circ a_1)\big(\hat{\mathcal O}_{B_{13}}\big)=&\;a_3^2\big(\hat{\mathcal O}_{B_{13}}\big)
\end{align*}
\smallskip\\
which by a sequence of commutativity relations (\ref{eq:MCGCommutativity}) and braid relations (\ref{eq:MCGBraid}) can be transformed into

\begin{align*}
(a_1\circ b_{12}\circ a_2)^4\big(\hat{\mathcal O}_{B_{13}}\big)=&\;a_3^2\big(\hat{\mathcal O}_{B_{13}}\big)
\end{align*}
\smallskip\\
Summarizing we conclude that

\begin{align*}
(a_1\circ b_{12}\circ a_2)^4=a_3^2
\end{align*}
\smallskip\\
as isomorphisms of $\mathcal A_H$ what finalizes the proof.

\end{proof}

\newpage

\section{Conclusion}

In this paper, we have constructed an algebra associated to a genus two surface much in the same way as spherical DAHA is associated to a torus. In this section, we would like to explain our motivation, the open problem(s) and how we see our rezult fits a more general picture.

Our main motivation comes from Chern-Simons topological quantum field theory. Namely, as explained in \cite{W,Quant}, Chern-Simons theory can be formulated in terms of quantization of the celebrated \emph{moduli space of flat G-connections on $\Sigma_g$}, where $G = SL(N,\mathbb C)$ and $\Sigma_g$ a closed genus $g$ surface. To specify more precisely what quantization means in this context, one can consider the fundamental group

\begin{align*}
\pi_1(\Sigma_g)=\left\{X_i,Y_i,\,1\leqslant i\leqslant g\;\big|\;[X_1,Y_1]\dots[X_g,Y_g]=1\right\}
\end{align*}
\smallskip\\
which is generated by paths along the $A$ and $B$ cycles and has a single relation. Fixing a point $p\in\Sigma_g$ and some choice of basic paths corresponding to generators of $\pi_1(\Sigma_g)$ which start and end at $p$, we get for each flat connection $\nabla$ on $\Sigma_g$ a matrix representation of $\pi_1(\Sigma_g)$,

\begin{align*}
\varphi(X_k)=\left(\begin{array}{ccc}
x_{11}^{(k)}&\dots&x_{1N}^{(k)}\\
\vdots&\ddots&\vdots\\
x_{N1}^{(k)}&\dots&x_{NN}^{(k)}
\end{array}\right)\qquad
\varphi(Y_k)=\left(\begin{array}{ccc}
y_{11}^{(k)}&\dots&y_{1N}^{(k)}\\
\vdots&\ddots&\vdots\\
y_{N1}^{(k)}&\dots&y_{NN}^{(k)}
\end{array}\right).
\end{align*}
\smallskip\\
All $SL(N,\mathbb C)$ representations of $\pi_1(\Sigma_g)$ then form an affine scheme $\mathcal V$ with a coordinate ring

\begin{align*}
\mathbb C[\mathcal V]=\mathbb C\big[x_{ij}^{(k)}\big]/J
\end{align*}
where $J$ is the ideal in $\mathbb C\big[x_{ij}^{(k)}\big]$ generated by all matrix elements of $\varphi([X_1,Y_1]\dots[X_g,Y_g])$ and $2g$ polynomials of the form $(\det\varphi(X_k)-1)$ and $(\det\varphi(Y_k)-1)$.

Gauge equivalent classes of flat connections are then in one-to-one correspondence with an isomorphism classes $\mathcal V_{\sim}$ of representations of $\pi_1(\Sigma_g)$. The latter has a coordinate ring $\mathbb C[\mathcal V]^{GL(N,\mathbb C)}\subset\mathbb C[\mathcal V]$ consisting of $GL(N,\mathbb C)$ invariant polynomials. It is well-known that, as a ring, $\mathbb C[\mathcal V]^{GL(N,\mathbb C)}$ is generated by traces $\left\{\textrm{tr}(\varphi(x))\;\big|\;x\in\pi_1(\Sigma_g)\right\}$. Of course, for a given $N$ there is only finitely many of them which are algebraically independent. In particular, for $N = 2$ we have a space of dimension

\begin{align*}
\mbox{dim}_{{\mathbb C}} \ \mathcal V_\sim = \left\{ \begin{array}{lll} 2, \ \ \ g = 1 \\ \\ 6 g - 6, \ \ \ g > 1 \end{array} \right.
\end{align*}
\smallskip\\
The coordinate ring of the moduli space of flat connections $\mathbb C[\mathcal V]^{GL(N,\mathbb C)}$ can be equipped with a Poisson bracket \cite{Poisson}. By a quantization of the moduli space of flat connections we mean a general associative algebra with a parameter $q$ which in the leading order expansion at $q=1$ reduces to a commutative algebra $\mathbb C[\mathcal V]^{GL(N,\mathbb C)}$, namely the coordinate ring of the moduli space of flat connections, whereas the next-to-leading order is controlled by a Poisson bracket on $\mathbb C[\mathcal V]^{GL(N,\mathbb C)}$. This algebra is the prototype and motivation for our construction.

Diffeomorphisms of $\Sigma_g$ act on the homology of the surface, in particular they act on the basic cycles generating the fundamental group and thus on holonomies of flat connection along these cycles. As a result they act on the moduli space of representations of the fundamental group. One has to note here that diffeomorphisms isotopic to identity act trivially on the conjugacy classes of holonomies of flat connection. Altogether this means that the above action actually factors through the action of the so-called mapping class group. The quantization algebra inherits this action in the following sence: the mapping class group is represented by automorphisms of the algebra. These automorphisms are usually outer.

To summarize, quantization of the moduli space of flat connections provides an algebra together with an action of the mapping class group by automorphisms, and in this sence is similar to the algebra that we construct. The crucial difference is that the quantization algebra has a single parameter $q$, while the algebra that we constructed has two parameters $q,t$ while still enjoying the full action of the mapping class group of a \textit{closed} genus two surface. In the case of the torus, the extra parameter $t$ is usually introduced via a puncture \cite{EK} which essentially does not change the mapping class group $SL(2,\mathbb Z)$, but in the case of genus two surface such a construction is not known, since introducing a puncture would change the mapping class group significantly.

Our current result therefore opens the following problem: does there exist an algebra of difference operators with two parameters $q,t$ which reduces to the quantization of the moduli space of flat connections on the genus $g$ surface for $q = t$ but enjoys the action of the mapping class group of a \textit{closed} surface by outer automorphisms for generic $q,t$? We argue that the algebra we constructed gives the solution to this problem for $g = 2$, much in the same way as spherical DAHA gives a solution for $g = 1$. The solution for $g > 2$, if it exists, would be interesting from several perspectives, including representation and knot theory.

Another open problem is to generalize the algebra we considered in this paper to higher rank $N > 2$ while keeping the genus $g = 2$. Yet another interesting problem is to investigate the classical limit of this algebra, where commutators are replaced by Poisson brackets. Last but not the least, an open problem is to find a non-symmetric version of this algebra. If it exists, this algebra would be a genus 2 analogue of non-spherical DAHA. Among other reasons, it would be interesting to find because many proofs and derivations are much simpler in the non-spherical setting, so one can reasonably expect the same to happen in the genus 2 case.

Finally, because of the action of the full genus 2 mapping class group, we expect that the algebra that we constructed should allow to compute $q,t$-invariants of arbitrary genus 2 knots, more general than can be computed by means of the conventional DAHA. We plan to address this topic in a follow-up publication.

\section*{Acknowledgments}

We would like to thank Prof. A.~Okounkov and Prof. V.~Rubtsov for fruitful discussions and Prof. I.~Cherednik for useful remarks. The work of S.A. was partly supported by the grants RFBR 15-01-04217 and 17-51-50051 Ya-F. The work of Sh.Sh. was partly supported by the grant RFBR 15-01-05990.

\newpage

\appendix

\section{Proof of Proposition \ref{propeigen}.}
\label{AppA}

\begin{lemma}
The operator ${\hat{\cal O}}_{A_1}$ acts on a product of conventional Macdonald polynomials via

\begin{align}
{\hat{\cal O}}_{A_1} \cdot P_{n}( x_{12} ) P_{m}( x_{13} ) P_{k}( x_{23} ) = \sum_{n^{\prime} = 0}^{n} \sum_{m^{\prime} = 0}^{m} \ \dfrac{g_{n}g_{m}}{g_{n^{\prime}}g_{m^{\prime}}} \ {\cal O}_{n,m|n^{\prime},m^{\prime}} \ P_{n^{\prime}}( x_{12} ) P_{m^{\prime}}( x_{13} ) P_{k}( x_{23} )
\label{MacBasis}
\end{align}
\smallskip\\
where

\begin{align*}
g_n = \dfrac{ \prod_{i = 0}^{n-1} \big( q^{\frac{i+1}{2}} - q^{\frac{-i-1}{2}} \big) }{ \prod_{i = 0}^{n-1} \big( q^{\frac{i}{2}} t^{\frac{1}{2}} - q^{\frac{-i}{2}} t^{\frac{-1}{2}} \big) }
\end{align*}
\smallskip\\
is the combinatorial norm of Macdonald polynomials, and

\begin{align*}
{\cal O}_{n,m|n^{\prime},m^{\prime}} \ = \
\left\{ \begin{array}{llll}
q^{\frac{n^{\prime}+m^{\prime}}{2}} t^{\frac{1}{2}} + q^{-\frac{n^{\prime}+m^{\prime}}{2}} t^{-\frac{1}{2}}, & n^{\prime} = n, m^{\prime} = m \\
\\
(q^{\frac{n^{\prime}}{2}} - q^{-\frac{n^{\prime}}{2}})
(t^{\frac{1}{2}} q^{\frac{m^{\prime}}{2}} - t^{-\frac{1}{2}} q^{-\frac{m^{\prime}}{2}})
, & n^{\prime} = n, m^{\prime} \in m - 2 {\mathbb N} \\
\\
(t^{\frac{1}{2}} q^{\frac{n^{\prime}}{2}} - t^{-\frac{1}{2}} q^{-\frac{n^{\prime}}{2}})
(q^{\frac{m^{\prime}}{2}} - q^{-\frac{m^{\prime}}{2}})
, & n^{\prime} \in n - 2 {\mathbb N}, m^{\prime} = m,  \\
\\
(t^{\frac{1}{2}} + t^{-\frac{1}{2}})
(t^{\frac{1}{2}} q^{\frac{n^{\prime}}{2}} - t^{-\frac{1}{2}} q^{-\frac{n^{\prime}}{2}})
(t^{\frac{1}{2}} q^{\frac{m^{\prime}}{2}} - t^{-\frac{1}{2}} q^{-\frac{m^{\prime}}{2}})
, & n^{\prime} \in n - 2 {\mathbb N}, \ m^{\prime} \in m - 2 {\mathbb N} \\
\\
- (x_{23} + x_{23}^{-1}) (t^{\frac{1}{2}} q^{\frac{n^{\prime}}{2}} - t^{-\frac{1}{2}} q^{-\frac{n^{\prime}}{2}})
(t^{\frac{1}{2}} q^{\frac{m^{\prime}}{2}} - t^{-\frac{1}{2}} q^{-\frac{m^{\prime}}{2}}), & n^{\prime} \in n - 2 {\mathbb N} - 1, \ m^{\prime} \in m - 2 {\mathbb N} - 1 \\
\\
0, & \mbox{ otherwise }
\end{array} \right.
\end{align*}
\smallskip\\
\end{lemma}
\begin{proof}
Note that, since ${\hat{\cal O}}_{A_1}$ contains no shifts in $x_{23}$, it suffices to prove the case $k = 0$. The proof is based on the generating function for Macdonald polynomials,

\begin{align*}
\sum\limits_{n = 0}^{\infty} \ g_n^{-1} \ P_n(x) \ y^n = \dfrac{\big( q^{\frac{1}{2}} t^{-\frac{1}{2}} y x; q \big)_{\infty}\big( q^{\frac{1}{2}} t^{-\frac{1}{2}} y x^{-1}; q \big)_{\infty}}{\big( q^{\frac{1}{2}} t^{\frac{1}{2}} y x; q \big)_{\infty}\big( q^{\frac{1}{2}} t^{\frac{1}{2}} y x^{-1}; q \big)_{\infty}} := \Omega(x,y)
\end{align*}
\smallskip\\
where both $(y;q)_{\infty}$ and $(y;q)_{\infty}^{-1}$ are formal power series in $y$:

\begin{align*}
(y;q)_{\infty} = \prod\limits_{n=0}^{\infty} ( 1 - q^n y ) = \sum\limits_{n = 0}^{\infty} \dfrac{y^n}{q^n - 1}, \ \ \ \ \
\dfrac{1}{(y;q)_{\infty}} = \prod\limits_{n=0}^{\infty} ( 1 - q^n y )^{-1} = \sum\limits_{n = 0}^{\infty} \dfrac{q^{n(n-1)/2} y^n}{1 - q^n}
\end{align*}
\smallskip\\
Multiplying both sides of eq.(\ref{MacBasis}) by $g_n^{-1} g_m^{-1} y_{12}^n y_{13}^m$ and summing over $n,m \geq 0$, we get

\begin{align}
\nonumber & \sum\limits_{a,b \in \{\pm 1\}} \ a b \ \dfrac{(1 - t^{\frac{1}{2}} x_{23}^{+1} x_{12}^a x_{13}^b)(1 - t^{\frac{1}{2}} x_{23}^{-1} x_{12}^a x_{13}^b)}{t^{\frac{1}{2}} x_{12}^{a} x_{13}^b (x_{12}^{+1} - x_{12}^{-1})(x_{13}^{+1} - x_{13}^{-1})} \ \Omega(q^{\frac{a}{2}} x_{12}, y_{12})\Omega(q^{\frac{b}{2}} x_{13}, y_{13}) = \emph{} \\
\nonumber \\
\nonumber & = t^{\frac{1}{2}} \Omega(x_{12}, q^{\frac{1}{2}} y_{12})\Omega(x_{13}, q^{\frac{1}{2}} y_{13}) + t^{-\frac{1}{2}} \Omega( x_{12}, q^{-\frac{1}{2}} y_{12})\Omega(x_{13}, q^{-\frac{1}{2}} y_{13}) + \emph{} \\
\nonumber \\
\nonumber & + \dfrac{y_{13}^2}{1 - y_{13}^2} \sum\limits_{a,b} \ a b \ t^{\frac{a}{2}} \Omega(x_{12}, q^{\frac{a}{2}} y_{12}) \Omega( x_{13}, q^{\frac{b}{2}} y_{13}) + \dfrac{y_{12}^2}{1 - y_{12}^2} \sum\limits_{a,b} \ a b \ t^{\frac{b}{2}} \Omega(x_{12}, q^{\frac{a}{2}} y_{12}) \Omega( x_{13}, q^{\frac{b}{2}} y_{13}) + \emph{} \\
\nonumber \\
& + \dfrac{(t^{\frac{1}{2}} + t^{-\frac{1}{2}}) y_{12}^2 y_{13}^2 - (x_{23} + x_{23}^{-1}) y_{12} y_{13}}{(1 - y_{12}^2)(1 - y_{13}^2)} \sum\limits_{a,b} \ a b \ t^{\frac{a+b}{2}} \Omega(x_{12}, q^{\frac{a}{2}} y_{12}) \Omega(x_{13}, q^{\frac{b}{2}} y_{13})
\label{SeriesMacBasis}
\end{align}
\smallskip\\
as an equality in $\mathbf k(x_{12},x_{23},x_{13})[[y_{12},y_{13}]]$ i.e. in formal power series in $y_{12},y_{13}$ whose coefficients are rational functions in $x_{12},x_{23},x_{13}$. Dividing both sides by $\Omega(x_{12}, q^{\frac{1}{2}} y_{12}) \Omega(x_{13}, q^{\frac{1}{2}} y_{13})$, using the properties

\begin{align*}
\Omega(q^{\frac{a}{2}} x,y) = \dfrac{1 - t^{\frac{1}{2}} y x^a}{1 - t^{-\frac{1}{2}} y x^a} \ \Omega(x,q^{\frac{1}{2}} y), \ \ \ \ \ \
\Omega(x,q^{-\frac{1}{2}} y) = \dfrac{(1 - t^{-\frac{1}{2}} y x)(1 - t^{-\frac{1}{2}} y x^{-1})}{(1 - t^{\frac{1}{2}} y x)(1 - t^{\frac{1}{2}} y x^{-1})} \ \Omega(x,q^{\frac{1}{2}} y)
\end{align*}
\smallskip\\
and multiplying by the least common denominator, we obtain an equality in $\mathbf k(x_{12},x_{23},x_{13})[y_{12},y_{13}]$ whose proof is a direct computation in $\mathbf k(x_{12},x_{23},x_{13})$. Since it holds iff the original equality (\ref{SeriesMacBasis}) holds, this completes the proof.
\end{proof}

\begin{lemma}
A product of two conventional Macdonald polynomials in $x_{12}$ and $x_{13}$ resp. is decomposed in the basis of genus two Macdonald polynomials in the following way:

\begin{align}
P_{n}( x_{12} ) P_{m}( x_{13} ) = \sum\limits_{s = 0}^{\lfloor \frac{n+m}{2} \rfloor } \ \dfrac{{\cal N}_{n+m - 2 s, n, m}}{P_n(t^{\frac{1}{2}})P_m(t^{\frac{1}{2}})} \ \Psi_{n+m - 2 s, n, m}(x_{12},x_{13},x_{23})
\label{DecFormula}
\end{align}
\smallskip\\
where ${\cal N}$ are the product coefficients for conventional Macdonald polynomials,

\begin{align*}
P_{n}( x ) P_{m}( x ) = \sum\limits_{s = 0}^{\lfloor \frac{n+m}{2} \rfloor } \ {\cal N}_{n+m - 2 s, n, m} \ P_{n+m-2s}(x)
\end{align*}
\smallskip\\
given explicitly by

\begin{align}
{\cal N}_{n+m - 2 s, n, m} = \prod\limits_{i = 0}^{s-1} \dfrac{(1 - q^{n - i})(1 - q^{m - i})(1 - t^2 q^{n + m - 2 s + i})(1 - t q^i)}{(1 - t q^{n - i - 1})(1 - t q^{m - i - 1})(1 - t q^{n + m - 2 s + i + 1})(1 - q^{i + 1})}
\label{qtVerlinde}
\end{align}
\smallskip\\
\end{lemma}

\begin{proof}
The proof is by induction in $m$. The base case $m = 0$ is simple: in this case, the second factor in the numerator of ${\cal N}_{n+m - 2 s, n, m}$ vanishes unless $s = 0$, which gives

\begin{align*}
P_{n}( x_{12} ) = \dfrac{\Psi_{n, n, 0}(x_{12},x_{13},x_{23})}{P_n(t^{\frac{1}{2}})}
\end{align*}
\smallskip\\
which we already proved in Corollary \ref{Reduction}. Next, assume (\ref{DecFormula}) holds for all $m \leq \ell$. Then,

\begin{align*}
P_{n}( x_{12} ) P_{\ell+1}( x_{13} ) = (x_{13} + x_{13}^{-1}) P_{n}( x_{12} ) P_{\ell}( x_{13} ) - \dfrac{(1 - q^\ell)(1 - q^{\ell - 1} t^2)}{(1 - q^\ell t)(1 - q^{\ell-1} t)} P_{n}( x_{12} ) P_{\ell-1}( x_{13} ) = \\
\\
=  \sum\limits_{s = 0}^{\lfloor \frac{n+\ell+1}{2} \rfloor } \ \dfrac{{\cal N}_{n+\ell - 2 s, n, \ell}}{P_n(t^{\frac{1}{2}})P_\ell(t^{\frac{1}{2}})} \ \sum\limits_{a,b \in \{\pm 1\}} \ C_{a,b}(n+\ell - 2 s, \ell, n) \ \Psi_{n+\ell - 2 s + a, n, \ell + b}(x_{12},x_{13},x_{23}) - \\
\\
- \dfrac{(1 - q^\ell)(1 - q^{\ell - 1} t^2)}{(1 - q^\ell t)(1 - q^{\ell-1} t)} \sum\limits_{s = 0}^{\lfloor \frac{n+\ell-1}{2} \rfloor } \ \dfrac{{\cal N}_{n + \ell - 2 s - 1, n, \ell - 1}}{P_n(t^{\frac{1}{2}})P_{\ell-1}(t^{\frac{1}{2}})} \ \Psi_{n+\ell - 2 s-1, n, \ell-1}(x_{12},x_{13},x_{23}) =\\
\\
= \sum\limits_{s = 0}^{\lfloor \frac{n+\ell+1}{2} \rfloor } \ \alpha_{s} \ \Psi_{n+\ell - 2 s + 1, n, \ell + 1}(x_{12},x_{13},x_{23}) + \sum\limits_{s = 0}^{\lfloor \frac{n+\ell-1}{2} \rfloor } \ \alpha^{\prime}_{s} \ \Psi_{n+\ell - 2 s - 1, n, \ell - 1}(x_{12},x_{13},x_{23})
\end{align*}
\smallskip\\
where $\alpha_s$ and $\alpha^{\prime}_s$ are shorthand notations for

\begin{align*}
\alpha_s =
\dfrac{{\cal N}_{n + \ell - 2 s, n, \ell}}{P_n(t^{\frac{1}{2}})P_\ell(t^{\frac{1}{2}})} \ C_{+1,+1}(n + \ell - 2 s, \ell, n) +
\dfrac{{\cal N}_{n + \ell - 2 s + 2, n, \ell}}{P_n(t^{\frac{1}{2}})P_\ell(t^{\frac{1}{2}})} \ C_{-1,+1}(n+\ell - 2 s + 2, \ell, n)
\end{align*}
\begin{align*}
\alpha^{\prime}_s \ = \ &
\dfrac{{\cal N}_{n + \ell - 2 s, n, \ell}}{P_n(t^{\frac{1}{2}})P_\ell(t^{\frac{1}{2}})} \ C_{-1,-1}(n + \ell - 2 s, \ell, n) +
\dfrac{{\cal N}_{n + \ell - 2 s - 2, n, \ell}}{P_n(t^{\frac{1}{2}})P_\ell(t^{\frac{1}{2}})} \ C_{+1,-1}(n+\ell - 2 s - 2, \ell, n) \\
\\
& \emph{} - \dfrac{(1 - q^\ell)(1 - q^{\ell - 1} t^2)}{(1 - q^\ell t)(1 - q^{\ell-1} t)} \ \dfrac{{\cal N}_{n + \ell - 2 s - 1, n, \ell - 1}}{P_n(t^{\frac{1}{2}})P_{\ell-1}(t^{\frac{1}{2}})}
\end{align*}
\smallskip\\
A direct computation with the explicit formulas eq.(\ref{qtVerlinde}) for ${\cal N}$ and eq.(\ref{Cab}) for $C$ shows that

\begin{align*}
\alpha_s = \dfrac{{\cal N}_{n + \ell - 2 s + 1, n, \ell + 1}}{P_n(t^{\frac{1}{2}})P_{\ell+1}(t^{\frac{1}{2}})}, \ \ \ \ \ \alpha^{\prime}_s = 0
\end{align*}
\smallskip\\
what completes the proof.
\end{proof}

\pagebreak

\begin{lemma}
Let ${\hat{\cal W}}_{A_1}$ be an operator on ${\cal H}$ defined by
\begin{align*}
\hat{\cal W}_{A_1}\Psi_{j_1,j_2,j_3}=\left(q^{\frac {j_1}2}t^{\frac12}+q^{-\frac {j_1}2}t^{-\frac12}\right)\Psi_{j_1,j_2,j_3}
\end{align*}
\smallskip\\
Then,

\begin{align}
{\hat{\cal W}}_{A_1} \cdot P_{n}( x_{12} ) P_{m}( x_{13} ) P_{k}( x_{23} ) = \sum_{n^{\prime} = 0}^{n} \sum_{m^{\prime} = 0}^{m} \ \dfrac{g_{n}g_{m}}{g_{n^{\prime}}g_{m^{\prime}}} \ {\cal O}_{n,m|n^{\prime},m^{\prime}} \ P_{n^{\prime}}( x_{12} ) P_{m^{\prime}}( x_{13} ) P_{k}( x_{23} )
\label{MacBasis2}
\end{align}
\smallskip\\
with the same $g_n$ and ${\cal O}_{n,m|n^{\prime},m^{\prime}}$ as in eq.(\ref{MacBasis}).
\end{lemma}
\begin{proof}
Note that, since multiplication by a function of $x_{23}$ commutes with $\hat{\cal W}_{A_1}$ (by the Pieri rule eq.(\ref{Pieri23})) it suffices to prove the case $k = 0$. Expanding both sides of the equality in the basis of genus two Macdonald polynomials, we find

\begin{align*}
& \sum\limits_{s = 0}^{\lfloor \frac{n+m}{2} \rfloor } \ \dfrac{{\cal N}_{n+m - 2 s, n, m}}{P_n(t^{\frac{1}{2}})P_m(t^{\frac{1}{2}})} \ \left(q^{\frac {n+m - 2 s}2}t^{\frac12}+q^{-\frac {n+m - 2 s}2}t^{-\frac12}\right) \Psi_{n+m - 2 s, n, m}(x_{12},x_{13},x_{23}) = \emph{} \\
\nonumber \\
& \emph{} = \sum_{n^{\prime} = 0}^{n} \sum_{m^{\prime} = 0}^{m} \ \dfrac{g_{n}g_{m}}{g_{n^{\prime}}g_{m^{\prime}}} \ {\cal O}_{n,m|n^{\prime},m^{\prime}} \ \sum\limits_{s = 0}^{\lfloor \frac{n^{\prime}+m^{\prime}}{2} \rfloor } \ \dfrac{{\cal N}_{n^{\prime} + m^{\prime} - 2 s, n^{\prime}, m^{\prime}}}{P_{n^{\prime}}(t^{\frac{1}{2}})P_{m^{\prime}}(t^{\frac{1}{2}})} \ \Psi_{n^{\prime}+m^{\prime} - 2 s, n^{\prime}, m^{\prime}}(x_{12},x_{13},x_{23})
\end{align*}
\smallskip\\
Equating coefficients in front of $\Psi_{n+m - 2 s, n, m}(x_{12},x_{13},x_{23})$, we find an identity

\begin{align*}
& \dfrac{{\cal N}_{n+m - 2 s, n, m}}{P_n(t^{\frac{1}{2}})P_m(t^{\frac{1}{2}})} \ \left(q^{\frac {n+m - 2 s}2}t^{\frac12}+q^{-\frac {n+m - 2 s}2}t^{-\frac12}\right) \ = \ \emph{}\\
\\
& \emph{} \ = \ \dfrac{{\cal N}_{n + m - 2 s, n, m}}{P_{n}(t^{\frac{1}{2}})P_{m}(t^{\frac{1}{2}})} \left(q^{\frac {n+m}2}t^{\frac12}+q^{-\frac {n+m}2}t^{-\frac12}\right) + \dfrac{g_{n}g_{m}}{g_{n - 1}g_{m - 1}} \dfrac{{\cal N}_{n + m - 2 s, n - 1, m -1 }}{P_{n-1}(t^{\frac{1}{2}})P_{m-1}(t^{\frac{1}{2}})} C_{++}(n - 1, m - 1, n + m - 2 s)
\end{align*}
\smallskip\\
whose proof is a direct computation in the field, with the explicit formulas eq.(\ref{qtVerlinde}) for ${\cal N}$ and eq.(\ref{Cab}) for $C$. Similarly, equating coefficients in front of $\Psi_{n^{\prime}+m^{\prime} - 2 s, n^{\prime}, m^{\prime}}(x_{12},x_{13},x_{23})$ with any given $s$ and $(n^{\prime}, m^{\prime}) \neq (n,m)$ results in an identity whose proof is a direct computation in the field, as well, completing the proof.
\end{proof}

\begin{customprop}{\ref{propeigen}}
Genus two Macdonald polynomials are common eigenfunctions of ${\hat {\cal O}}_{A_1}, {\hat {\cal O}}_{A_2}, {\hat {\cal O}}_{A_3}$
\begin{align*}
\hat{\mathcal O}_{A_k}\Psi_{j_1,j_2,j_3}=\left(q^{\frac {j_k}2}t^{\frac12}+q^{-\frac {j_k}2}t^{-\frac12}\right)\Psi_{j_1,j_2,j_3}
\end{align*}
\end{customprop}
\begin{proof}
It is enough to prove the statement for $\hat{\mathcal O}_{A_1}$, the other two will follow by $S_3$ symmetry (\ref{eq:PsiS3Symmetry}). Comparing (\ref{MacBasis}) and (\ref{MacBasis2}) we conclude that ${\hat{\cal O}}_{A_1}$ and ${\hat{\cal W}}_{A_1}$ are identical operators on ${\cal H}$.
\end{proof}

\pagebreak

\end{document}